\documentclass{amsart}
\usepackage[utf8]{inputenc}

\usepackage{amssymb}
\usepackage{amsmath}
\usepackage{amsfonts}
\usepackage{tikz-cd}
\usepackage{amsthm}
\usepackage{enumitem}
\usepackage{mathrsfs}
\usetikzlibrary{decorations.pathmorphing}
\usepackage[utf8]{inputenc}
\usepackage[T1]{fontenc}
\usepackage{indentfirst}

\newtheorem{theorem}{Theorem}[section]
\newtheorem{lemma}[theorem]{Lemma}

\theoremstyle{definition}
\newtheorem{define}[theorem]{Definition}
\newtheorem{example}[theorem]{Example}

\newtheorem{proposition}[theorem]{Proposition}

\theoremstyle{remark}
\newtheorem{remark}[theorem]{Remark}

\numberwithin{equation}{section}

\DeclareMathOperator{\Hom}{Hom}
\DeclareMathOperator{\Ext}{Ext}

\DeclareMathOperator{\Der}{Der}

\DeclareMathOperator{\ann}{ann}
\DeclareMathOperator{\Sp}{Sp}
\DeclareMathOperator{\gr}{gr}

\DeclareMathOperator{\Var}{V}
\DeclareMathOperator{\image}{im}
\DeclareMathOperator{\pdim}{pdim}
\DeclareMathOperator{\depth}{depth}
\DeclareMathOperator{\Spec}{Spec}
\DeclareMathOperator{\Supp}{Supp}
\DeclareMathOperator{\rel}{rel}
\DeclareMathOperator{\Ch}{Ch}
\DeclareMathOperator{\im}{im}
\DeclareMathOperator{\pr}{pr}

\usepackage [english]{babel}
\usepackage [autostyle, english = american]{csquotes}
\MakeOuterQuote{"}

\author{Daniel Bath}
\title{A noncommutative analogue of the Peskine--Szpiro Acyclicity Lemma}
\thanks{The author is supported by FWO grant \#G097819N and FWO grant \#12E9623N.}
\begin{document}
\sloppy
\maketitle
\begin{abstract}
    We present a variant of the Peskine--Szpiro Acyclicity Lemma, and hence a way to certify exactness of a complex of finite modules over a large class of (possibly) noncommutative rings. Specifically, over the class of Auslander regular rings. In the case of relative $\mathscr{D}_{X}$-modules, for example $\mathscr{D}_{X}[s_{1}, \dots, s_{r}]$-modules, the hypotheses have geometric realizations making them easier to authenticate. We demonstrate the efficacy of this lemma and its various forms by: independently recovering some results related to Bernstein--Sato polynomials; establishing a new result about quasi-free structures of free multi-derivations of hyperplane arrangements.
\end{abstract}

\section{Introduction}

The goal of this note is to find a noncommutative analogue of the famous Peskine--Szpiro Acyclicity Lemma:

\begin{lemma} \text{\normalfont (Peskine--Szpiro Acyclicity Lemma \cite{PeskineSzpiroAcyclicity})}
Suppose $R$ is a commutative, Noetherian, local ring and $0 \to M_{q} \to M_{q-1} \to \cdots \to M_{0}$ a complex of finite $R$-modules such that $\depth M_{j} \geq j$. If the first nonvanishing homology module $H_{i}$ of this complex occurs at $i > 0$, then $\depth H_{i} \geq 1$.  
\end{lemma}

\noindent In practice, this is often used to verify exactness of a complex, not to learn about the depth of nonvanishing cohomology modules. 

The proof of Peskine--Szpiro's lemma has two main ingredients. First: work with the homological characterization of depth, that is, the index of the first nonvanishing $\Ext_{R}^{k}(R/\mathfrak{m}, -)$ module; break up the complex into short exact sequences; obtain lower bounds on depth for every module that appears in these sequences. Second: where exactness fails, use the homological characterization of depth to deduce: if a module has positive depth, so do any of its (nonzero) submodules. The second ingredient depends on the fact that our homological interpretation of depth arises from the covariant functor $\Hom_{R}(R/\mathfrak{m}, -)$. This blunt appeal to a covariant functor is why the lemma's conclusion ``$\depth H_{i} \geq 1$'' has the indelicate lower bound of $1$.

In a noncommutative setting, we cannot appeal to local-ness nor depth. By Auslander--Buchsbaum it is reasonable to instead work with projective dimension. So we do. Our homological characterization will then involve the vanishing of $\Ext_{R}^{k}(-, R)$ and, as such, arise from the contravariant functor $\Hom_{R}(-,R)$. This means the second ingredient used in the sketch above is inapplicable. Double dual $\Ext$-modules, i.e. $\Ext_{R}^{i}(\Ext_{R}^{j}(-, R), R)$, will help address this deficiency. 

Our lemma is for Auslander regular rings $A$ (Definition \ref{def- Auslander regular ring}) which are known to have well behaved dual $\Ext$ modules $\Ext^{k}_{A}(-, A)$, and well behaved double dual $\Ext$ modules $\Ext_{A}^{i}(\Ext_{A}^{j}(-, A), A)$. These rings contain the commutative, regular, Noetherian rings (Proposition \ref{prop- Bjork commutative auslander regular facts}) but also many other rings, especially noncommutative ones. The formalism of such rings has recently proved fertile in the realm of $\mathscr{D}_{X}$-modules, $\mathscr{D}_{X}[s_{1}, \dots, s_{r}]$-modules, and questions involving Bernstein--Sato polynomials and ideals. (The rings mentioned are Auslander regular.) Maisonobe realized their utility in \cite{Maisonobe} and Budur, van der Veer, Wu, and Zhou used \cite{ZeroLociI} some of his ideas to prove a deep relation between Bernstein--Sato ideals and local systems on the complement of a divisor. We recommend Bj\"{o}rk's Appendices in \cite{Bjork} for a primer on these rings, although this note is essentially self-sufficient. 

Before charting the path of this brief, let us state our noncommutative analogue of the Peskine--Szpiro Acyclicity Lemma. Note that the conclusion is in terms of the grade $j(N)$ of a module (Definition \ref{def- grade}), that is, the index of the first nonvanishing dual $\Ext$ module $\Ext_{A}^{k}(N, A)$. The following appears with proof as Lemma \ref{lemma- noncommutative acyclity lemma Auslander}:

\begin{lemma} \label{lemma - intro noncommutative Auslander regular acyclicity lemma}
\text{\normalfont (Noncommutative Peskine--Szpiro Acyclicity Lemma)}
Let $A$ be an Auslander regular ring and 
\[
M_{\bullet} := 0 \to M_{m} \to M_{m-1} \to \cdots \to M_{1} \to M_{0}
\]
a complex of finitely generated left $A$-modules such that $\pdim(M_{q}) \leq m-q$. Denote by $H_{q}$ the homology at slot $q$. Let $i > 0$ be the largest index so that $H_{i} \neq 0$. Then $j(H_{i}) \leq m - i < m.$
\end{lemma}

\vspace{2mm}

The paper's outline is thus: In section 2, we first define Auslander regular rings, present the minimal pieces from the Appendices of \cite{Bjork} in order to state Lemma \ref{lemma- noncommutative acyclity lemma Auslander}, our ``noncommutative Peskine--Szpiro Acyclicity Lemma,'' and then we prove it. In the original the conclusion gives a rough lower bound on the depth of the first nonvanishing homology module; in the new version, the conclusion gives a much finer upper bound on the grade of the first nonvanishing homology module. In subsection 2.2 we compare and contrast the two versions of the lemma, sketching their respective strengths and weaknesses.

In section 3 we restrict to relative $\mathscr{D}_{X}$-modules, (Definition \ref{def- relative D-modules}, e.g. $\mathscr{D}_{X} \otimes_{\mathbb{C}} R$, $R$ a localization of a polynomial ring over $\mathbb{C}$), which are instances of modules over Auslander regular rings. Here grade has a nice geometric interpretation via characteristic varieties, so we provide a geometric version of our lemma: Lemma \ref{lemma- algebraic noncommutative lemma} for the algebraic case; Lemma \ref{lemma-noncommutative acyclicity analytic} for the local, analytic. 

While our motivation was for relative $\mathscr{D}_{X}$-modules and Bernstein--Sato constructions, we have proved Lemma \ref{lemma- noncommutative acyclity lemma Auslander} in the broader context of Auslander regular rings in the hope it will be useful to a larger audience. In our experience, even in the $\mathscr{D}_{X}$-module setting, it is difficult to construct resolutions without appealing to associated graded arguments; Lemma \ref{lemma- noncommutative acyclity lemma Auslander} (and its children Lemmas \ref{lemma- algebraic noncommutative lemma}, \ref{lemma-noncommutative acyclicity analytic}) provide alternative infrastructure. 

In the final subsection 3.4 we demonstrate this in the realm of Bernstein--Sato polynomials and the $\mathscr{D}_{X,\mathfrak{x}}[s]$-module generated by $f^{s}$: in Proposition \ref{prop- free, Saito, Spencer resolution} we recover an important result of Calder{\'o}n Moreno's and Narv{\'a}ez Macarro's (Theorem 1.6.4 of \cite{MorenoMacarroLogarithmic}) by an entirely different argument; in Proposition \ref{prop-SpencerAcyclicMultiDerivations} we consider a variant of a complex considered in \cite{QuasiFree} in the case of free multi-derivations of hyperplane arrangements, establishing its acyclicity for the first time. The conclusion of Proposition \ref{prop- free, Saito, Spencer resolution} is used to understand deep questions about the singular structure of $f$: it appears, as outlined in Remark \ref{rmk - future use case, tame}, in many approaches to the Logarithmic Comparison Theorem, the Strong Monodromy Conjecture, and the $-n/d$ Conjecture. In Example \ref{ex - Reiffen plane curve} we use this Proposition to say something novel about a well-studied example of a problematic plane curve (as well as free and Saito-holonmic divisors in general). The utility of Proposition \ref{prop-SpencerAcyclicMultiDerivations} is unclear: its connection to the singularity structure of the original hyperplane arrangement is mostly unexplored.

\vspace{2mm}

We would like to thank Uli Walther for helpful conversations and his interest in and encouragement about the problem. We would also like to thank Francisco Castro Jim{\'e}nez and Luis Narv{\'a}ez Macarro for their support and for the shared mathematics that helped inspire this work. Finally, we are grateful to the referees for their comments and suggestions, all of which improved the manuscript.

\section{Our Lemma for Auslander Regular Rings}

We work with rings $A$ that are both left and right Noetherian as well as finitely generated left (and sometimes right) $A$-modules. We further assume that $A$ has \emph{finite global homological dimension}, that is there exists a $\mu \in \mathbb{N}$ such that for any finitely generated left $A$-module $M$, the projective dimension of $M$ satisfies $\pdim(M) \leq \mu$. The smallest $\mu$ satisfying this condition is the global dimension of $A$ and is denoted by $\text{gdim} A = \mu$.\footnote{Actually we have defined finite \emph{left} global homological dimension and \emph{left} global dimension. There are symmetric \emph{right} variants by considering finitely generated right $A$-modules $N$ instead. As we have assumed $A$ is both left and right Noetherian, the left and right notions coincide: its left and right global dimension agree.}

Before defining our class of rings, we need the following classical homological notion:

\begin{define} \label{def- grade}
Let $A$ be a Noetherian (both left and right) ring of finite global homological dimension. For $M$ a finitely generated left $A$-module, the \emph{grade} $j(M)$ of M is the smallest integer $j(M)$ such that $\Ext_{A}^{j(M)}(M, A) \neq 0$. A similar definition of grade applies for right $A$-modules.
\end{define}

With this in hand, we have enough to define the Auslander regular rings:

\begin{define} \label{def- Auslander regular ring}
A ring $A$ is an \emph{Auslander regular ring} provided it satisfies the following conditions:
\begin{enumerate}[label=(\alph*)]
    \item A is Noetherian (both left and right);
    \item A has finite global homological dimension; 
    \item A satisfies \emph{Auslander's condition}, i.e. for any finitely generated left $A$-module $M$ and for any submodule $N \subseteq \Ext_{A}^{k}(M,A)$, the grade of $N$ is bounded below by $j(N) \geq k$.
\end{enumerate}
\end{define}

As alluded to before, Bj\"{o}rk's Appendices, especially A:IV, of \cite{Bjork} give a simultaneously thorough and elementary treatment of Auslander regular rings. We only require one of his propositions and we note that it is a consequence of applying Auslander's condition to the double complex (and attendant spectral sequence) built by: taking a projective resolution $P_{\bullet}$ of $M$, looking at $\Hom_{R}(P_{\bullet}, R)$, and resolving each $\Hom$ module. Here is the proposition we need:

\begin{proposition} (Bj\"{o}rk A:IV, Proposition 2.2 \cite{Bjork}) \label{prop-Bjork grade of Ext modules}
Let $A$ be an Auslander regular ring and $M$ a finitely generated left $A$-module. Then 
\[
j(\Ext_{A}^{j(M)}(M, A)) = j(M).
\]
\end{proposition}

In particular, Proposition \ref{prop-Bjork grade of Ext modules} implies that for any finitely generated left module M over an Auslander-regular
ring A, we have
\begin{equation*}
    \Ext_A^{j(M)} (\Ext_A^{j(M)}(M, A), A) \neq 0.
\end{equation*}

\subsection{The Lemma}

With this minimal set-up we can now state and prove our noncommutative analogue of the Peskine--Szpiro Acyclicity Lemma:

\begin{lemma} \label{lemma- noncommutative acyclity lemma Auslander} \text{\normalfont (Noncommutative Peskine--Szpiro Acyclicity Lemma)}
Let $A$ be an Auslander regular ring and 
\[
M_{\bullet} := 0 \to M_{m} \to M_{m-1} \to \cdots \to M_{1} \to M_{0}
\]
a complex of finitely generated left $A$-modules such that $\pdim(M_{q}) \leq m-q$. Denote by $H_{q}$ the homology at slot $q$. Let $i > 0$ be the largest index so that $H_{i} \neq 0$. Then $j(H_{i}) \leq m - i < m.$
\end{lemma}
\begin{proof}

Label the differentials as $d_{q}: M_{q} \to M_{q-1}$ so that $H_{q} = \ker d_{q} / \image d_{q+1}$. We first assume $i \leq m -2 $. 

\emph{Step 1}: By exactness before $i$ we have the short exact sequences
\[
0 \to M_{m} \to M_{n-1} \to \image d_{m-1} \to  0 
\]
\[
0 \to \image d_{m-1} \to M_{m-2} \to \image d_{m-2} \to 0 \
\]
\[
\vdots
\]
\[
0 \to \image d_{i+2} \to M_{i+1} \to \image d_{i+1} \to 0. 
\]
Using the long exact sequence in $\Ext_{A}(-, A)$ and the fact $\pdim (M_{m}) \leq 0$, $\pdim (M_{m-1}) \leq 1$, the first short exact sequence implies $\pdim (\image d_{m-1}) \leq 1.$ Because $\pdim (M_{m-2}) \leq 2$, the same method applied to the second short exact sequence implies $\pdim (\image d_{m-2}) \leq 2$. Iterating, we deduce $\pdim (\image d_{i+1}) \leq m - (i+1)$.

\emph{Step 2}: From the inclusion $\ker d_{i} \hookrightarrow M_{i}$ we obtain a short exact sequence 
\[
0 \to \ker d_{i} \to M_{i} \to Q \to 0.
\]
Again consider the long exact sequence of $\Ext_{A}(-,A)$ modules. Since $\pdim (M_{i}) \leq m-i$, for all $t \geq m - i + 1$ we have $\Ext_{A}^{t}(\ker d_{i}, A) \simeq \Ext_{A}^{t+1}(Q, A)$. By Auslander's condition, we have a lower bound on grade: $j(\Ext_{A}^{t+1}(Q, A)) \geq t+1$. So 
\[
\Ext_{A}^{t}(\Ext_{A}^{t}(\ker d_{i}, A), A) \simeq \Ext_{A}^{t}(\Ext_{A}^{t+1}(Q, A), A) = 0  \text{ for all } t \geq m - i + 1.
\]

\emph{Step 3}: Now consider the short exact sequence 
\begin{equation} \label{eqn-step 3 ses}
0 \to \image d_{i+1} \to \ker d_{i} \to H_{i} \to 0.
\end{equation}
Using $\pdim (\image d_{i+1}) \leq m - (i+1)$ from Step 1 and inspecting the long exact sequence of $\Ext_{A}(-,A)$ modules reveals an isomorphism 
\[
\Ext_{A}^{t}(H_{i}, A) \simeq \Ext_{A}^{t}(\ker d_{i}, A) \text{ for all } t \geq m - (i+1) + 2 = m - i + 1
\]
By Step 2,
\[
\Ext_{A}^{t}(\Ext_{A}^{t}(H_{i}, A), A) = 0 \text{ for all } t \geq m - i + 1.
\]
If $j(H_{i}) \geq m - i + 1$ this yields
\[
\Ext_{A}^{j(H_{i})}(\Ext_{A}^{j(H_{i})}(H_{i}, A), A) = 0,
\]
which is impossible by Proposition \ref{prop-Bjork grade of Ext modules}. Therefore $j(H_{i}) \leq m - i.$

It remains to tackle the cases of $i =m, m-1.$ In both situations Step 2 goes through exactly as before. Now we argue similarly to Step 3, replacing the invocation of Step 1 with ad hoc arguments. If $i = m$, then $\ker d_m = H_m$. Then Step 2 implies $\Ext_A^t (\Ext_A^t (H_m, A), A) = 0$ for all $t \geq 1$, forcing $j(H_m) \leq 0$ as required. If $i = m-1$, then \eqref{eqn-step 3 ses} becomes $0 \to M_m \to \ker d_{m-1} \to H_{m-1} \to 0$. Because $\pdim M_m \leq 0$, we see that $\Ext_A^t(\ker d_{m-1}, A) \simeq \Ext_A^t(H_{m-1}, A)$ for all $t \geq 2$. Step 2 gives $\Ext_A^t(\Ext_A^t(H_{m-1}, A), A) = 0$ for all $t \geq 2$, implying that $j(H_{m-1}) \leq 1$ and completing the proof.
\end{proof}

\subsection{The commutative analogue}

Lemma \ref{lemma- noncommutative acyclity lemma Auslander} looks very similar to Peskine--Szpiro's Acyclicity Lemma. However, even in the commutative case the two results are somewhat skew. To see this, we quote more results from Appendix IV of \cite{Bjork}. But first recall that if $B$ is a commutative Noetherian ring and $M$ a left $B$-module, the dimension of $M$ is 
\begin{equation*}
    \dim(M) = \min \{ \text{height} (\mathfrak{p}) \mid \mathfrak{p} \in \Spec(B) \text{ is a minimal prime of } \ann_R(M) \}
\end{equation*}
where $\ann_B(M) = (0 :_B M)$ is the $B$-annihilator of $M$.

\begin{proposition} \label{prop- Bjork commutative auslander regular facts} (Bj\"{o}rk Proposition A:IV.3.4 \cite{Bjork}, Proposition 4.5.1 of \cite{ZeroLociI}) 
Let $A$ be a commutative, regular, Noetherian ring. Then:
\begin{enumerate}[label=(\alph*)]
\item $A$ is an Auslander regular ring;
\item If $M \neq 0$ is a finitely generated left $A$-module and if $\dim A_{\mathfrak{m}}$ is the same for every maximal ideal $\mathfrak{m} \subseteq A$, then
\[
j(M) + \dim(M) = \dim A_{\mathfrak{m}}.
\]
\end{enumerate}
\end{proposition}

First of all, Peskine--Szpiro only requires that $A$ be a commutative, local, Noetherian ring. We may use Proposition \ref{prop- Bjork commutative auslander regular facts} to find a commutative version of Lemma \ref{lemma- noncommutative acyclity lemma Auslander}, but this requires the extra assumption of regularity on $A$. Additionally, Peskine--Szpiro assumes $\depth M_{i} \geq i$. By Auslander--Buschbaum this gives the upper bound $\pdim(M_{i}) \leq \depth A - i$. So the hypotheses on projective dimension between the two lemmas are only compatible if we assume the index $m$ in Lemma \ref{lemma- noncommutative acyclity lemma Auslander} satisfies $m = \depth A$.

If we assume enough to resolve these discrepancies (commutative, local, regular, $\dim A_{\mathfrak{m}} = m$ the length of $M_{\bullet}$), the conclusions of the two lemmas differ thanks to: (1) the difference between dimension and depth; (2) the difference in the bounds each lemma gives. The Peskine--Szpiro Acyclicity Lemma concludes with ``the first nonzero homology module $H_{i}$ has positive depth''; by Proposition \ref{prop- Bjork commutative auslander regular facts}, the conclusion of Lemma \ref{lemma- noncommutative acyclity lemma Auslander} is ``the first nonzero homology module $H_{i}$ has dimension $\geq i$.'' On one hand, Peskine--Szpiro is stronger: positive depth implies positive dimension, but not vice versa. On the other hand, our noncommutative Peskine--Szpiro Acyclicity Lemma is more robust: the original Acyclicity Lemma cannot obtain as fine as a lower bound as ``$\dim H_{i} \geq i$'', instead it is limited to ``dimension is positive.''

\section{The case of relative $\mathscr{D}_{X}$-modules} 

Now we turn to the situation of relative $\mathscr{D}_{X}$-modules where $X$ is a complex manifold or smooth algebraic variety and $\mathscr{O}_{X}$ the algebraic structure sheaf. These rings turn out to be Auslander regular (Proposition \ref{prop-Bjork associated graded commutative}) and so the previous facts are salient. Being geometric, this setting comes with a visual enrichment of Lemma \ref{lemma- noncommutative acyclity lemma Auslander}. We first work with algebraic differential operators $\mathscr{D}_{X}$ and later discuss how to extend the results to the local analytic setting. So assume everything is algebraic until subsection 3.3.

\begin{define} \label{def- relative D-modules} 
For $R$ a finitely generated $\mathbb{C}$-algebra that is a regular integral domain, denote
\[
\mathscr{A}_{R} = \mathscr{D}_{X} \otimes_{\mathbb{C}} R.
\]
A module over $\mathscr{A}_{R}$ is a \emph{relative} $\mathscr{A}_{R}$-module; sometimes we call these \emph{relative} $\mathscr{D}_{X}$-module.
\end{define}

The famous order filtration $F_{\bullet}$ on $\mathscr{D}_{X}$ induces a filtration 
\[
F_{i} \mathscr{A}_{R} = F_{i} \mathscr{D}_{X} \otimes_{\mathbb{C}} R
\]
on $\mathscr{A}_{R}$ which we call the \emph{relative order filtration}. Let $\gr^{\text{rel}}\mathscr{A}_{R} = \gr \mathscr{D}_{X} \otimes_{\mathbb{C}} R$ be the associated graded ring. Note that this is a positive filtration (see next subsection), $\gr^{\text{rel}}\mathscr{A}_{R}$ is a commutative, regular, Noetherian ring, and $\Spec \gr^{\text{rel}} \mathscr{A}_{R} = T^{\star} X \times \Spec R$ where $T^{\star}X$ is the cotangent bundle.

For any finitely generated left $\mathscr{A}_{R}$-module, good filtrations (see next subsection) exist with respect to the relative order filtration. Before proceeding, we detour with some classical facts about filtrations and associated graded objects. We can use the results about positively filtered rings here since the relative order filtration is a particularly nice positive filtration. 

\subsection{Positively filtered rings} 
Suppose $A$ is \emph{positively filtered}, that is, has a filtration of subgroups $F_{j} A \subseteq F_{j+1}$ indexed by the integers such that: (1) the filtration is compatible with multiplication in the natural way; (2) $F_{-1} A = 0$. Let $\gr^{F} A = \oplus F_{j+1} A / F_{j} A$ be the associated graded ring. Now let $M$ be a finitely generated $A$-module with a filtration $\Gamma$. We say $\Gamma$ is a \emph{good filtration} when $\Gamma$ is bounded below and $\gr M = \gr^\Gamma M = \oplus \Gamma_v M / \Gamma_{v-1}$ is a finitely generated $\gr^F$-module (Proposition A:III.1.29 \cite{Bjork}).
\begin{define}
Suppose that $A$ is a positively filtered ring such that $\gr^F A$ is commutative. Let $M$ be a finitely generated left $A$-module equipped with a good filtration. Denote the associated graded $\gr^{F}$-module $\gr M$. Then the \emph{characteristic ideal} of $M$ is
\[
J^{F}(M) = \text{rad} ( \ann_{\gr^{F}A} \gr M)
\]
and this does not depend on the choice of good filtration of $M$, cf. A:III.3.20 \cite{Bjork}. The \emph{characteristic variety} of $M$ is 
\[
\Ch^{F}(M) = \Var( J^{F}(M))
\]
which also does not depend on the choice of good filtration.
\end{define}

Assuming the existence of a positive filtration is useful because often $\gr^{F} A$ is a commutative, regular, Noetherian ring. In this case

\begin{proposition} \label{prop-Bjork associated graded commutative} (Bj\"{o}rk A:IV.5.1, A:IV.4.15 \cite{Bjork})
Suppose that $A$ is a positively filtered ring such that $\gr^{F} A$ is a commutative, regular, Noetherian ring. Then:
\begin{enumerate}[label=(\alph*)]
    \item $A$ is Auslander regular;
    \item if $M \neq 0$ is a finitely generated left $A$ module, $j(M) = j(\gr M)$. 
\end{enumerate}
\end{proposition}

\subsection{Our lemma for algebraic, relative $\mathscr{D}_{X}$-modules}

Using the characteristic variety arising from the relative order filtration, we may formulate parts of Lemma \ref{lemma- noncommutative acyclity lemma Auslander} in the geometric picture of a module's $\mathscr{O}_{X}$-support. Here we work with algebraic $\mathscr{D}_{X}$, recalling that from Proposition \ref{prop-Bjork associated graded commutative} the relative order filtration reveals that $\mathscr{A}_{R}$ is Auslander regular.

\begin{lemma} \label{lemma- algebraic noncommutative lemma}
Let $\dim X = n$ and let
\[
M_{\bullet} := 0 \to M_{n} \to M_{n-1} \to \cdots \to M_{1} \to M_{0}
\]
a complex of finitely generated left $\mathscr{A}_{R}$-modules such that $\pdim(M_{q}) \leq n-q$. Denote by $H_{q}$ the homology at slot $q$. If, for each $i>0$, the $\mathscr{O}_{X}$-support of each homology module $H_{i}$ has dimension at most $i - 1$, then $M_{\bullet}$ is acyclic. In particular, the augmented complex resolves $M_{0} / \im (M_{1} \to M_{0})$.
\end{lemma}
\begin{proof}
Let $i > 0$ be the largest index such that $H_{i} \neq 0$. By Lemma \ref{lemma- noncommutative acyclity lemma Auslander}, $j(H_{i}) \leq n- i < n - i + 1.$ By Proposition \ref{prop-Bjork associated graded commutative}, $j(\gr H_{i}) < n - i + 1$ as well. Since $\gr^{\rel} \mathscr{A}_{R}$ is a finitely generated $\mathbb{C}$-algebra, all its maximal ideals have the same height $2n + \dim R$. Proposition \ref{prop- Bjork commutative auslander regular facts} applies and we discover $\dim \Ch^{\rel}(H_{i}) > n + i - 1 + \dim R.$

On the other hand, since $T^{\star}X \times \Spec R$ is conical in the $\gr(\partial)$ direction, the canonical fibration $\pi: T^{\star}X \to X$ extends to a fibration $(\pi,\text{id}): T^{\star}X \times \Spec R \to X \times \Spec R$ that is the identity on the second factor. In particular, this lets us approximate the support of $\gr^{\text{rel}} N$ for any finitely generated left $\mathscr{A}_{R}$-module $N$:
\[
\Ch^{\rel} (N) \subseteq (\pi, \text{id})^{-1}(\text{Supp}_{\mathscr{O}_{X} \otimes_{\mathbb{C}} R} N) \subseteq \pi^{-1} (\text{Supp}_{\mathscr{O}_{X}} N) \times \Spec R.
\]
In the case of $H_{i} = N$, we know that $H_{i}$ does not vanish only $\Supp_{\mathscr{O}_{X}} H_{i} \subseteq X$, and we also know, by hypothesis, this support has dimension at most $i - 1$. So
\[
\Ch^{\rel}(H_{i}) \subseteq \pi^{-1}(\Supp_{\mathscr{O}_{X}} H_{i}) \times \Spec R,
\]
and since the fibers of $\pi^{-1}(\mathfrak{x})$ are each conic Lagrangians of dimension $n$, we see that $\dim \Ch^{\rel} (H_{i}) \leq n + i - 1 + \dim R$. This conflicts with the lower bound from the first paragraph. 

Thus for all $i > 0$ the homology modules $H_{i}$ must vanish.
\end{proof}

\begin{remark} \label{rmk- algebraic vanish outside of points simplification}
Note that this result is still interesting if we simplify the assumption about $\mathscr{O}_{X}$-support to: each homology module $H_{i}$ vanishes outside a discrete set of $X$. Ideally this would be the case in an inductive set-up.
\end{remark}

\subsection{Our lemma for analytic, relative $\mathscr{D}_{X}$-modules}
Now let $X$ be a smooth analytic space or $\mathbb{C}$-scheme of dimension $n$, $\mathscr{O}_{X}$ the analytic structure sheaf, and $\mathscr{D}_{X}$ the sheaf of $\mathbb{C}$-linear differential operators over $\mathscr{O}_{X}$. Let $\mathscr{A}_{R}$ be defined as before. A similar result to Lemma \ref{lemma- noncommutative acyclity lemma Auslander} holds in this setting, provided we work at $\mathfrak{x} \in X$ and so with complex of finitely generated $\mathscr{A}_{R,\mathfrak{x}}$-modules.

We still have the same relative order filtration on $\mathscr{A}_{R,\mathfrak{x}}$. The main issue is the assumption in part (b) of Proposition \ref{prop- Bjork commutative auslander regular facts} may not hold. In fact, there are maximal ideals of different heights in $\mathscr{O}_{X, \mathfrak{x}}[s]$, which of course doesn't happen in the algebraic setting. We may skirt this issue by using partial analytification and the theory of graded modules and graded maximal ideals. Using the relative order filtration (and maximal graded ideals of the associated graded object) lets us use these two notions in tandem. The methodology is outlined in good detail in section 3.6 of \cite{ZeroLociI}, though we (and they) note its origin comes from ideas of Maisonobe in \cite{Maisonobe}.

What we need is the following crystallization of their cumulative labor: if $M$ is a finite $\mathscr{A}_{R,\mathfrak{x}}$-module then
\begin{equation} \label{eqn - analytic grade plus dim fml}
    j( \gr M) + \dim (\gr M) = \dim (\gr^{\rel} A_{R,\mathfrak{x}}) = 2n + \dim R.
\end{equation}

Using \eqref{eqn - analytic grade plus dim fml} to buttress the first paragraph of the proof of Lemma \ref{lemma- noncommutative acyclity lemma Auslander}, the same argument gives:

\begin{lemma} \label{lemma-noncommutative acyclicity analytic}
Let $\dim X = n$ and let 
\[
M_{\bullet} := 0 \to M_{n} \to M_{n-1} \to \cdots \to M_{1} \to M_{0}
\]
a complex of finitely generated left $\mathscr{A}_{R, \mathfrak{x}}$-modules such that $\pdim(M_{q}) \leq n-q$. Denote by $H_{q}$ the homology at slot $q$. If, for each $i>0$, the homology module $H_{i}$ vanishes at all points $\mathfrak{y}$ near $\mathfrak{x}$ outside a set of dimension at most $i-1$, then $M_{\bullet}$ is acyclic. In particular, the augmented complex resolves $M_{0} / \im (M_{1} \to M_{0})$.
\end{lemma}

\begin{remark} \label{rmk- analaytic vanish outside of points simplification}
Again: this is still of interest under the simplifying assumption: each homology module $H_{i}$ vanishes at all points $\mathfrak{y} \neq \mathfrak{x}$ near $\mathfrak{x}$.
\end{remark}

\subsection{Use cases, actual and potential}

Here we demonstrate an application of this noncommutative Peskine--Szpiro Acyclicity Lemma in the realm of Bernstein--Sato polynomials. A crash course is required.

Stay in the setting of the preceding subsection. So $\mathscr{O}_{X}$ is analytic and $\dim X = n$. Take $f$ to be a reduced global section (reduced is not necessary but is simpler); $s$ a dummy variable. By formally taking derivatives of $f^{s}$ using the chain rule, we can regard $\mathscr{O}_{X}[s, 1/f] \otimes f^{s}$ as a $\mathscr{D}_{X}[s]$-module. Do so and let $\mathscr{D}_{X}[s] f^{s}$ be the cyclic submodule generated by the symbol $f^{s}$. Being cyclic, we can define
\[
M_{f} = \frac{\mathscr{D}_{X}[s]}{\ann_{\mathscr{D}_{X}[s]} f^{s}} \simeq \mathscr{D}_{X}[s] f^{s}.
\]
The famous \emph{Bernstein--Sato polynomial} of $f$ at $\mathfrak{x}$ is the monic generator of 
\[
\ann_{\mathbb{C}[s]} \left( \frac{\mathscr{D}_{X,\mathfrak{x}}[s] f^{s}}{\mathscr{D}_{X,\mathfrak{x}}[s] f^{s+1}} \right)
\]
and is known to be nonzero.

Let $\Der_{X}(-\log f) = \{ \delta \in \Der_{X} \mid \delta \bullet f \in \mathscr{O}_{X} \cdot f \}$ be the \emph{logarithmic derivations} of $f$; this is an $\mathscr{O}_{X}$-module closed under brackets. Each logarithmic derivation $\delta$ corresponds to annihilating element of $f^{s}$ by 
\begin{equation}  \label{eqn - map giving theta}
\Der_{X}(-\log f) \ni \delta \mapsto \delta - s \frac{\delta \bullet f}{f} \in \ann_{\mathscr{D}_{X}[s]}f^{s}.
\end{equation}
In practice, these specific annihilators are the only easily producible ones. Let $\theta_{f}(s)$ be $\mathscr{O}_{X}[s]$-submodule generated by the image of the map \eqref{eqn - map giving theta}. Consider two new modules defined by the following natural short exact sequence:
\begin{equation} \label{eqn - the log ses}
\begin{tikzcd}
0 \rar
    & \frac{\ann_{\mathscr{D}_{X}[s]} f^{s}}{\mathscr{D}_{X}[s] \cdot \theta_{f}(s)} \rar \dar[equals]
        & \frac{\mathscr{D}_{X}[s]}{\mathscr{D}_{X}[s] \cdot \theta_{f}(s)} \rar \dar[equals]
            & \frac{\mathscr{D}_{X}[s]}{\ann_{\mathscr{D}_{X}[s]} f^{s}} \rar \dar[equals]
                & 0 \\
0 \rar 
    & K_{f}^{\log} \rar
        & M_{f}^{\log} \rar
            & M_{f} \rar
                & 0.
\end{tikzcd}
\end{equation}

As $M_{f}^{\log}$ has an explicit description, it is easier to study; in desirable settings $M_{f}^{\log} \simeq M_{f}$, making the Bernstein--Sato polynomial (and $M_{f}$) much easier to understand and/or compute. We now define a natural candidate to resolve $M_{f}^{\log}$, following the treatment of Calder{\'o}n Moreno and Narv{\'a}ez Macarro in \cite{MorenoMacarroLogarithmic, LNMSpencer}; see also the appendix of Narv{\'a}ez Macarro's \cite{MAcarroDuality}.

\begin{define} \label{def - log spencer complex} (see subsection 1.1.8 of \cite{MorenoMacarroLogarithmic}, A.18 of \cite{MAcarroDuality})
The \emph{Spencer complex} (Cartan-Eilenburg-Chevalley-Rinehart-Spencer complex) $\Sp_f^\bullet$ associated to $\theta_{f}(s)$ is a complex of $\mathscr{D}_{X}[s]$-modules defined as follows. The objects are left $\mathscr{D}_{X}[s]$-modules (the action comes from acting on the left of $\otimes$)
\[
\Sp_f^{-r} = \mathscr{D}_{X}[s] \otimes_{\mathscr{O}_{X}[s]} \bigwedge^{r} \theta_{f}(s)
\]
and the $\mathscr{D}_{X}[s]$-differential $\epsilon^{-r}: \Sp_f^{-r} \to \Sp_f^{-(r+1)}$ is given by 
\[
\epsilon^{-r} (P \otimes \lambda_{1} \wedge \cdots \wedge \lambda_{r}) = \sum_{i} (-1)^{r-1} (P \lambda_{i}) \otimes \widehat{\lambda_{i}} + \sum_{1 \leq i \leq j \leq r} (-1)^{i +j} P \otimes [\lambda_{i}, \lambda_{j}] \wedge \widehat{\lambda_{i, j}},
\]
under the convention $\widehat{\lambda_{i}}$ is the initial ordered wedge product excluding $\lambda_{i}$ and $\widehat{\lambda_{i,j}}$ is the initial ordered wedge product excluding both $\lambda_{i}$ and $\lambda_{j}$. For $r=1$ we have
\[
\epsilon^{-1} (P \otimes \lambda) = P \lambda \in \Sp_f^{0} = \mathscr{D}_{X}[s],
\]
making the augmented complex
\[
\Sp_f^\bullet \to M_{f}^{\log}.
\]
\end{define}

We want to reprove the ``hard part'' (i.e. the part not using ``differential linear type'') of Calder{\'o}n Moreno and Narv{\'a}ez Macarro's Theorem 1.6.4 of \cite{MorenoMacarroLogarithmic} using our noncommutative Peskine--Szpiro Acyclicity Lemma. This proposition, in our language, essentially says that $\Sp_f^\bullet$ resolves $M_{f}^{\log}$ under certain hypothesis on the logarithmic derivations of $f$. It is significant for many reasons, one being that, to our knowledge, it is the only systematic way to find a such a resolution.

\begin{define} \label{def - hypotheses on log derivations} \text{\normalfont (Common working hypotheses)}
We say $f$ is \emph{free at} $\mathfrak{x}$ if $\Der_{X,\mathfrak{x}}(-\log f)$ is a free $\mathscr{O}_{X,\mathfrak{x}}$-module; $f$ is simply $\emph{free}$ if it is so at all $\mathfrak{x} \in X$.  

We may stratify $X$ using the logarithmic derivations as follows. Two points $\mathfrak{x}$ and $\mathfrak{y}$ belong to the same class equivalence class if there is some open $U \ni \mathfrak{x}, \mathfrak{y}$, a $\delta \in \Der_{U}(-\log f)$, that is (i) nowhere vanishing on $U$ and (ii) whose integral curve passes through both $\mathfrak{x}$ and $\mathfrak{y}.$ The transitive closure of this relation gives the \emph{logarithmic stratification} of $X$. We say $f$ is \emph{Saito-holonomic} when this stratification is locally finite. 
\end{define}

\begin{remark} \label{rmk- details of working hypotheses} \text{ }
\begin{enumerate}[label=(\alph*)]
    \item Freeness and Saito-holonomicity originated in \cite{SaitoLogarithmicForms}. We recommend section 2 of \cite{uli} (especially Remark 2.6) and Lemma 3.2, 7.3 of \cite{FormalStructure} for the technical details underlying what follows.
    \item Freeness and Saito-holonomicity are independent of the choice of defining equation of $f$. So is $\Der_{X,\mathfrak{x}}(-\log f)$.
    \item The Saito-holonomic property allows an inductive argument. It is always true that: if $\mathfrak{x}$ belongs to a positive dimensional stratum $\sigma$, then there is a local analytic isomorphism between the pairs $(X, \text{div}(f))$ and $\mathbb{C}^{\dim \sigma} \times (X^{\prime}, \text{div}(f^{\prime}))$ where $\dim X^{\prime} = n - \dim \sigma$. Under the Saito-holonomic assumption, one can induce on the dimension of $X$, use the inductive hypothesis and the above to assume $\mathfrak{x}$ belongs to a zero dimensional stratum, and use Saito-holonomicity to assume that a desired property holds at all $\mathfrak{y}$ near $\mathfrak{x}$. (See Remark 2.6 of \cite{uli}.)
    \item To make the above item a bit more explicit, we summarize Lemma 3.2 and Lemma 7.3 of \cite{FormalStructure}. When $\mathfrak{x}$ belongs to a positive dimensional stratum $\sigma$ then there is a local analytic isomorphism $(X, \text{div}(f), \mathfrak{x}) \simeq (\mathbb{C}^{\dim \sigma}, \mathbb{C}^{\dim \sigma},0) \times (X^{\prime}, \text{div}(f^{\prime}), \mathfrak{x}^{\prime})$. Let $\pr_i$ be the projection of $X$ onto the $i^{\text{th}}$ factor using the above product structure. Then we have
    \[
    \Der_{X,\mathfrak{x}}(-\log f) = \pr_1^\star \Der_{\mathbb{C}^{\text{dim} \sigma}, 0} \oplus \pr_2^\star \Der_{X, \mathfrak{x}^\prime}(-\log f^\prime).
    \]
    In particular $\Der_{X, \mathfrak{x}}(\log f) = \Der_{X, \mathfrak{x}}(\log \pr_2^\star f^\prime)$. Moreover elements of $\pr_1^\star \Der_{\mathbb{C}^{\text{dim} \sigma}, 0}$ act on $\pr_2^\star f^\prime$ as multiplication by $0$; the commutator of an element of $\pr_1^\star \Der_{\mathbb{C}^{\text{dim} \sigma}, 0}$ with an element of $\pr_2^\star \Der_{X, \mathfrak{x}^\prime}(-\log f^\prime)$ vanishes. 
    \item Stay in the setting and notation of (d). Writing $(\mathbb{C}^{\dim \sigma}, \mathbb{C}^{\dim \sigma},0)$ as $(\mathbb{C}, \mathbb{C}, 0) \times (\mathbb{C}^{\dim \sigma -1}, \mathbb{C}^{\dim \sigma -1}, 0)$ we see there is a germ $h \in \mathscr{O}_{Y, y}$ such that  
    \[
    \Der_{X, \mathfrak{x}} (-\log \pr_2^\star h) = \Der_{X,\mathfrak{x}}(-\log f) = \pr_1^\star \Der_{\mathbb{C}, 0} \oplus \pr_2^\star \Der_{Y,\mathfrak{y}}(-\log h)
    \]
    where $(Y, \text{div}(g), \mathfrak{y}) \simeq (\mathbb{C}^{\dim \sigma -1}, \mathbb{C}^{\dim \sigma -1}, 0) \times (X^\prime, \text{div}(f^\prime), \mathfrak{x}^\prime).$ This simplification of (d) is helpful for induction.
\end{enumerate}
\end{remark}

Now we reprove the ``hard part'' of Theorem 1.6.4 from \cite{MorenoMacarroLogarithmic} by a logically distinct argument--one that avoids appeal to associated graded objects and regular sequences. (Note that for free divisors, Saito-holonomicity is equivalent to being Koszul free.) First a lemma.

\begin{lemma} \label{lem-externalTensor}
    Suppose that $X = \mathbb{C} \times Y$ and let $\pr_2: X \to Y$ be the projection onto the second factor. Fix $h \in \mathscr{O}_{Y}$. If $\Sp_h^\bullet$ is acylic, then $\Sp_{\pr_2^\star h}^\bullet$ is acylic.
\end{lemma}

\begin{proof}
    Let $\pr_1 : X \to \mathbb{C}$ be the projection onto the first factor. Recall that for a (left) $\mathscr{D}_{\mathbb{C}}$-module $M$ and a (left) $\mathscr{D}_Y$-module $N$, the external tensor product is the $\mathscr{D}_{X}$-module
    \begin{equation*}
        M \boxtimes N = \mathscr{D}_X \otimes_{\pr_1^\star \mathscr{D}_{\mathbb{C}} \otimes_\mathbb{C} \pr_2^\star \mathscr{D}_Y} (\pr_1^\star M \otimes_\mathbb{C} \pr_2^\star N).
    \end{equation*}
    Moreover, $\boxtimes$ is exact in both arguments (pg 38-39 \cite{HTT}) and so extends to a map between bounded derived categories $D^b(\mathscr{D}_{\mathbb{C}}) \times D^b(\mathscr{D}_Y) \xrightarrow[]{\boxtimes} D^b(\mathscr{D}_X).$ Abusing notation, for a (left) $\mathscr{D}_{\mathbb{C}}[s]$-module $A$ and a (left) $\mathscr{D}_{Y}[s]$-module $B$ we denote the $\mathscr{D}_X[s]$-module
    \begin{equation*}
        A \boxtimes B = \mathscr{D}_X[s] \otimes_{\pr_1^\star \mathscr{D}_{\mathbb{C}}[s] \otimes_\mathbb{C} \pr_2^\star \mathscr{D}_Y[s]} (\pr_1^\star A \otimes_\mathbb{C} \pr_2^\star B).
    \end{equation*}
    Since polynomial ring extensions are flat, this external tensor also extends to a map between bounded derived categories $D^b(\mathscr{D}_{\mathbb{C}}[s]) \times D^b(\mathscr{D}_Y[s]) \xrightarrow[]{\boxtimes} D^b(\mathscr{D}_X[s]).$ 

    We will construct an acyclic complex of $\mathscr{D}_{\mathbb{C}}[s]$-modules $G^\bullet$ such that $\Sp_{\pr_2^\star h}^\bullet \simeq G^\bullet \boxtimes \Sp_h^\bullet$. Then by the above, in the bounded derived category $D^b(\mathscr{D}_X[s])$ the augmented complex $\Sp_{\pr_2^\star h}^\bullet \to M_{\pr_2^\star h}^{\log}$ will be represented by the external tensor of a single $\mathscr{D}_{\mathbb{C}}[s]$-module with the single $\mathscr{D}_{Y}[s]$-module $M_{h}^{\log}$. Hence $\Sp_{\pr_2^\star h}^\bullet$ is acylic and the proof will be complete.
    
    Let $G^\bullet = [0 \to G^{-1} \to G^0]$ be the complex of $\mathscr{D}_{\mathbb{C}}[s]$-modules given by $G^{-1} =  \mathscr{D}_{\mathbb{C}}[s] \otimes_{\mathscr{O}_{\mathbb{C}}[s]} \Der_{\mathbb{C}}[s]$ and $G^0 = \mathscr{D}_{\mathbb{C}}[s]$. The nontrivial differential is given by $P \otimes \delta \mapsto P \delta$. Fix a generator $\alpha$ of $\Der_{\mathbb{C}}[s]$. Then our $G^\bullet$ is isomorphic to the complex $0 \to \mathscr{D}_{\mathbb{C}}[s] \xrightarrow[]{\cdot \alpha} \mathscr{D}_{\mathbb{C}[s]}$ induced by right multiplication with $\alpha$. As $\mathscr{D}_{\mathbb{C}}[s]$ is torsion-free,  $G^\bullet$ is acylic.

    It remains to show that $\Sp_{\pr_2^\star h}^\bullet \simeq G^\bullet \boxtimes \Sp_h^\bullet$. We may do this locally, so pick $(t, \partial_t)$ as coordinates for $\mathbb{C}$ and $(y, \partial y)$ as coordinates for $Y$. The conjunction of their pullbacks give coordinates for $X$. Since $\pr_2^\star h$ only uses $y$-coordinates, $\Der_{X}(\log \pr_2^\star h) \simeq \pr_1^\star \Der_{\mathbb{C}} \oplus \pr_2^\star \Der_{Y}(\log h)$ and 
    \begin{align*} \label{eqn-extTensor1}
    \bigwedge^r \Der_{X}(\log \pr_2^\star h)[s] \simeq 
    &\big(\pr_1^\star \Der_{\mathbb{C}}[s] \wedge \bigwedge^{r-1} \pr_2^\star \Der_{Y}(\log h)[s] \big) \\
    &\oplus \big( \pr_1^\star \mathscr{O}_\mathbb{C}[s] \wedge \bigwedge^r \pr_2^\star \Der_{Y}[s](\log h) \big). \nonumber
    \end{align*}
    It follows that $\Sp_{\pr_2^\star h}^{-r} \simeq \big( G^{-1} \boxtimes \Sp_{h}^{-(r-1)} \big) \oplus \big( G^0 \boxtimes \Sp_h^{-r} \big)$. We must confirm this isomorphism is compatible with the various differentials to deduce $\Sp_{\pr_2^\star h}^\bullet \simeq G^\bullet \boxtimes \Sp_h^\bullet$. Examining the definition of differential $\epsilon^{-r}$ of $\Sp_{\pi_2^\star h}^\bullet$, this reduces to checking that the ``mixed term'' $\sum_{1 \leq j \leq r-1}(-1)^{2+j} 1 \otimes  [\pr_1^\star \partial_t, \pr_2^\star \lambda_j] \wedge \widehat{\pr_2^\star \lambda_j}$ vanishes for any $\lambda_1, \dots, \lambda_{r-1} \in \Der_Y(\log h)$. But this is immediate: each commutator $[\pr_1^\star \partial_t, \pr_2^\star \lambda_j]$ vanishes since $\partial_t$ uses only $(t, \partial_t)$-variables and $\lambda_j$ uses only $(y, \partial y)$-variables along with the central variable $s$. 
\end{proof}

And now the promised result:

\begin{proposition} \text{\normalfont (Theorem 1.6.4 of \cite{MorenoMacarroLogarithmic})} \label{prop- free, Saito, Spencer resolution}
Suppose that $f$ is Saito-holonomic and free. Then at each $\mathfrak{x} \in X$, the augmented complex
\[
\Sp_f^\bullet \to M_{f}^{\log}
\]
is a free $\mathscr{D}_{X,\mathfrak{x}}[s]$-resolution of $M_{f,\mathfrak{x}}^{\log}.$
\end{proposition}

\begin{proof}
First note that since $\Der_{X,\mathfrak{x}}(-\log f)$ is free by hypothesis, all we must do is show $\Sp_f^\bullet$ is acyclic at $\mathfrak{x}.$ This can be proved in local coordinates. We are free to pick a different choice of defining equation of $\text{div}(f)$ at $\mathfrak{x}$ since $\Sp_f^\bullet$ is well-behaved with respect to replacing $f$ with $(\text{unit}) \cdot f$. 

Proceed by induction on $\dim X$. When $\dim X = 1$, then after a change of coordinates and of defining equation, we may assume $\mathfrak{x} = 0$, our $\text{div}(f)$ at $0$ is defined by $x$, and $\Der_{X,0}(-\log x) = \mathscr{O}_{X,0} \cdot x \partial$. Then $\Sp_f^\bullet$ at $\mathfrak{x}$ corresponds to the complex $0 \to \mathscr{D}_{X,0}[s] \to \mathscr{D}_{X,0}[s]$ with the second map given by right multiplication with $x \partial - s$. So $\Sp_f^\bullet$ is acyclic and we are done with the base case. 

If the claim holds whenever $\dim X < n$, assume $\dim X = n$. If $\mathfrak{x}$ belongs to a positive dimensional stratum $\sigma$, then thanks to Remark \ref{rmk- details of working hypotheses}.(e) (see the last item in particular), we can find an appropriate local analytic isomorphism so that $(X, \text{div}(f), \mathfrak{x}) \simeq (\mathbb{C}, \mathbb{C}, 0) \times (Y, \text{div}(h), \mathfrak{y})$ and so
\begin{equation*}
    \Der_{X, \mathfrak{x}} (-\log \pr_2^\star h) = \Der_{X,\mathfrak{x}}(-\log f) = \pr_1^\star \Der_{\mathbb{C}, 0} \oplus \pr_2^\star \Der_{Y,\mathfrak{y}}(-\log h).
\end{equation*}
Replacing $X$ and $Y$ with sufficiently small opens, we may assume the above isomorphism holds globally. Since $h \in \mathscr{O}_X$ is free, this isomorphism forces $h \in \mathscr{O}_Y$ to be free. Moreover, since $f$ is Saito-holonomic, so is $h$. Thus $h$ and $\Sp_h^\bullet$ succumb to our inductive hypothesis: $\Sp_h^\bullet$ is an acyclic complex of free $\mathscr{D}_Y[s]$-modules. Using Lemma \ref{lem-externalTensor}, we deduce that $\Sp_f^\bullet$ is acylic at $\mathfrak{x}$. 

All that remains is the case where $\mathfrak{x}$ belongs to a zero dimensional logarithmic stratum. By Saito-holonomicity, we can assume $\Sp_f^\bullet$ is acyclic at all points near, but distinct from, $\mathfrak{x}$. This means that for $i \neq 0$, the homology modules $H_{i} \Sp_f^\bullet$ vanish at all points near, but distinct from, $\mathfrak{x}$. Using Lemma \ref{lemma-noncommutative acyclicity analytic} (see also Remark \ref{rmk- analaytic vanish outside of points simplification}), we conclude $\Sp_f^\bullet$ is acyclic. 
\end{proof}

\begin{remark} \label{rmk - future use case, tame} \text{ }
\begin{enumerate} [label=(\alph*)]
    \item In Theorem 1.6.4 of \cite{MorenoMacarroLogarithmic}, the language ``Koszul free'' is used, i.e. the principal symbols of $\delta - s (\delta \bullet f)/f$ (see \eqref{eqn - map giving theta}) form a regular sequence in the appropriate associated graded ring. While it is known that Koszul free is equivalent to free and Saito-holonomic, our proof of Proposition \ref{prop- free, Saito, Spencer resolution} avoids this fact or any knowledge about the behavior of $\Sp_f^\bullet$ under associated graded constructions. In this way we hope it's proof is more versatile. 
    \item Note that $f$ being reduced is not necessary in the construction of $\Sp_f^\bullet$ and is not used in the proof of Proposition \ref{prop- free, Saito, Spencer resolution}.
    \item Since Proposition \ref{prop- free, Saito, Spencer resolution} is just a demonstration, we have used the less cumbersome notation associated with $f^{s}$ and Bernstein--Sato polynomials instead of $F^{S} = f_{1}^{s_{1}} \cdots f_{r}^{s_{r}}$ and multivariate Bernstein--Sato polynomials. We know the same version of Proposition \ref{prop- free, Saito, Spencer resolution} holds in this case (see Proposition 5 of \cite{MaisonobeFreeHyperplanes}, and the generalization within Theorem 3.9 of \cite{Bath2}). The same argument of Proposition \ref{prop- free, Saito, Spencer resolution} gives a new proof in the multivariate setting.
    \item When $M_{f,\mathfrak{x}}^{\log} \simeq M_{f,\mathfrak{x}}$ it is much easier to find roots of the Bernstein--Sato polynomial of $f$. In most stories, a crucial step in verifying this isomorphism is proving $\Ext_{\mathscr{D}_{X,\mathfrak{x}}[s]}^{k}(M_{f,\mathfrak{x}}^{\log}, \mathscr{D}_{X,\mathfrak{x}}[s])$ vanishes for $k > n$. Proposition \ref{prop- free, Saito, Spencer resolution} does this. (This happens in \cite{MorenoMacarroLogarithmic} and is part of the set-up in \cite{uli}, where it happens on the associated graded side.) When this isomorphism holds for $f$ a hyperplane arrangement, the $-n/d$ conjecture and the Strong Monodromy Conjecture hold for $f$ by the argument of Theorem 5.13 of \cite{uli}.
    \item It would be very interesting to find a version of Proposition \ref{prop- free, Saito, Spencer resolution} where freeness is replaced by \emph{tameness}. A divisor $f$ is tame if the logarithmic $k$-forms have projective dimension at most $k$ (see Definition 3.8 of \cite{uli}, though the notion did not originate here). In fact, the sliding projective dimension bound of Peskine--Szpiro Acyclity Lemma (and hence of the noncommutative variant) align with the ones used in tameness. 
    \item While tameness and Saito-holonomicity are the purview of the salient parts of \cite{uli}, Walther also mandates ``strongly Euler-homogeneous'' so that associated graded constructions can be used. A ``tame version'' of Proposition \ref{prop- free, Saito, Spencer resolution} might unlock the non-strongly Euler-homogeneous case.
    \item Proposition \ref{prop- free, Saito, Spencer resolution} plays a role in the well-studied Logarithmic Comparison Theorem, c.f. \cite{MorenoMacarroLogarithmic} for a description of the problem. A ``modern'' approach involves specializing the acyclic complex $\Sp$ at $s = -1$, see Corollary 4.5 of \cite{MAcarroDuality}. Again, a ``tame version'' of Proposition \ref{prop- free, Saito, Spencer resolution} would give a path for studying the Logarithmic Comparison Theorem for tame divisors.
\end{enumerate}

\end{remark}

\begin{example} \label{ex - Reiffen plane curve}
Let $f = x^{4} + xy^{4} + y^{5}$. Since $X = \mathbb{C}^{2}$ has dimension $2$, our $f$ is automatically Saito-holonomic and free. So Proposition \ref{prop- free, Saito, Spencer resolution} automatically applies. However, Macaulay2 insists that a generating set of $\ann_{\mathscr{D}_{\mathbb{C}^{2}, 0}[s]} f^{s}$ has order two operators. (Or use Theorem 4.7 of \cite{MAcarroDuality}.) So $M_{f,0}^{\log} \twoheadrightarrow M_{f, 0}$ is not injective and the map's kernel $K_{f,0}^{\log}$ is nonzero. Using elementary facts about relative holonomicity (Definition 3.2.3 of \cite{ZeroLociI}, i.e. major{\'e} par
une lagrangienne from \cite{Maisonobe}), the long exact sequence arising from applying $\Hom_{\mathscr{D}_{\mathbb{C}^{2},0}[s]}(-, \mathscr{D}_{\mathbb{C}^{2},0}[s])$ to the short exact sequence \eqref{eqn - the log ses}, and Proposition \ref{prop- free, Saito, Spencer resolution}, we deduce 
\begin{equation} \label{eqn- ext modules logarithmic kernel}
\Ext_{\mathscr{D}_{\mathbb{C}^{2},0}[s]}^{k}(K_{f, 0}^{\log}, \mathscr{D}_{\mathbb{C}^{2},0}[s]) \neq 0 \iff k = \dim X = 2.
\end{equation}
The same holds for free and Saito-holonomic divisors in general: either $K_{f,\mathfrak{x}}^{\log} = 0$ or \eqref{eqn- ext modules logarithmic kernel} holds (where just $k = \dim X$ is used.) 
\end{example}

\subsection{Use cases, quasi-free structures and multiarrangements} In \cite{QuasiFree}, the authors consider Spencer complexes similar to Definition \ref{def - log spencer complex} except: they replace $s$ with an element of $\mathbb{C}$; their complexes are assembled out of a special free $\mathscr{O}_X$-submodule of $\Der_X(-\log f)$ closed under brackets instead of being assembled out of $\Der_X(-\log f)$ (or, when adding $s$-data, a special free submodule of $\theta_f(s)$ closed under brackets). The special free Lie algebroid is called a \emph{quasi-free structure} in loc. cit. When a quasi-free structure exists the construction gives a nice complex of free modules, but it seems at least as hard to verify it is acyclic than in the original setting. And certifying acylicity is crucial for any potential applications of quasi-free structures. In this subsection we track and modify the construction of \cite{QuasiFree} for free multi-derivations of hyperplane arrangements, and use Lemma \ref{lemma-noncommutative acyclicity analytic} to prove the resulting Spencer-esque complex is acyclic. This result is new. 

Let us set up. In this subsection, $f \in R = \mathbb{C}[x_1, \dots, x_n]$ always defines a reduced hyperplane arrangement of $d$ hyperplanes equipped with a factorization $f = \ell_1 \cdots \ell_d$ into pairwise coprime linear polynomials. Recall the intersection lattice of $f$ is the poset 
\begin{equation*}
    \left\{ \bigcap_{k \in K} \Var(\ell_k) \mid \text{ for all } K \subseteq \{1, \dots, d\} \right\},
\end{equation*} 
whose members are called flats. 

For $\textbf{m} \in \mathbb{Z}_{\geq 1}^d$ we define the \emph{algebraic multi-derivations} of $f$ with respect to $\mathbf{m}$ to be
\begin{equation*}
    \Der_R(-\log f; \mathbf{m}) = \{ \delta \in \Der_R \mid \delta \bullet \ell_k \in R \cdot \ell_k^{m_k} \text{ for all } 1 \leq k \leq d\}.
\end{equation*}
Let $X = \mathbb{C}^n$ and $\mathscr{O}_X$ be the analytic structure sheaf. Analytifying $\Der_R(-\log f; \mathbf{m})$ we obtain the $\mathscr{O}_X$-module $\Der_X(-\log f; \mathbf{m})$, the \emph{analytic multi-derivations} of $f$ with respect to $\textbf{m}$. One checks that 
\begin{equation*}
    \Der_X(-\log f; \mathbf{m}) \subseteq \Der_X(-\log f ; \mathbf{1}) = \Der_X(-\log f).
\end{equation*}
One also checks that $\Der_X(-\log f; \mathbf{m})$ is closed under brackets. We say $\Der_X(-\log f; \mathbf{m})$ is \emph{free} when it is a free $\mathscr{O}_X$-module locally everywhere. In this case it satisfies conditions (a) and (b) of Definition 2.1 \cite{QuasiFree}. (That it satisfies condition (b) is Theorem 1.14 \cite{YoshinagaFreenessSurvey}). We have no idea whether or not condition (c) holds. In the spirit of \cite{QuasiFree}, such a $\Der_X(-\log f; \mathbf{m})$ is a \emph{weak quasi-free structure}.

Now we use the analytic multi-derivations $\Der_X(-\log f ; \textbf{m})$ to fashion a Spencer-esque complex ala Definition \ref{def - log spencer complex}\footnote{Mirroring Definition \ref{def - log spencer complex}, the construction does not require $\Der_X(-\log f; \textbf{m})$ to be free.}.  First, set $\theta_f(s ; \textbf{m})$ to be the $\mathscr{O}_X[s]$-module 
\begin{equation} \label{eqn-thetaMultiDerivations}
    \theta_f(s ; \textbf{m}) = \mathscr{O}_X[s] \cdot \left\{ \delta - s \frac{\delta \bullet f}{f^{\textbf{m}}} \mid \delta \in \Der_X(-\log f; \textbf{m}) \right\}.
\end{equation}
That $(\delta \bullet f) / f^{\textbf{m}} \in \mathscr{O}_X$ uses the definition of multi-derivations. Also note that $\theta_f(s ; \textbf{m})$ is closed under brackets by construction. We are ready for the complex:

\begin{define}
    Let $f \in R$ define a hyperplane arrangement as before. Let $\Der_X(-\log f ; \textbf{m})$ be the module of analytic multi-derivations and $\theta_{f ; \textbf{m}}(s)$ the $\mathscr{O}_X[s]$-module of \eqref{eqn-thetaMultiDerivations}. The \emph{Spencer-esque complex} $\Sp_{f; \textbf{m}}^\bullet$ associated to $\theta_f(s ; \textbf{m})$ is a complex of left $\mathscr{D}_X[s]$-modules whose objects are
    \begin{equation*}
        \Sp_f^{-r} = \mathscr{D}_X[s] \otimes_{\mathscr{O}_X[s]} \bigwedge^r \theta_f(s; \textbf{m}). 
    \end{equation*}
    The left $\mathscr{D}_X[s]$-linear differential $\epsilon^{-r}: \Sp_f^{-r} \to \Sp_f^{-(r+1)}$ is defined analogously to the differential of Definition \ref{def - log spencer complex}. In particular, the augmented complex is
    \begin{equation*}
        \Sp_{f; \textbf{m}}^\bullet \to \frac{\mathscr{D}_X[s]}{\mathscr{D}_X[s] \cdot \theta_{f; \textbf{m}}(s)}. 
    \end{equation*}
\end{define}

When the analytic multi-derivations are free, we can prove $\Sp_{f ; \textbf{m}}^\bullet$ is acylic. As the proof is spiritually the same as the proof of Proposition \ref{prop- free, Saito, Spencer resolution} we only give a sketch.

\begin{proposition} \label{prop-SpencerAcyclicMultiDerivations}
    Let $f \in \mathbb{C}[x_1, \dots, x_n]$ be a hyperplane arrangement, $\textbf{m} \in \mathbb{Z}_{\geq 1}^d$, $X = \mathbb{C}^n$. Suppose that the analytic multi-derivations $\Der_X(-\log f; \textbf{m})$ are free. Then at each $\mathfrak{x} \in X$, the augmented complex
    \begin{equation*}
        \Sp_{f; \textbf{m}}^\bullet \to \frac{\mathscr{D}_X[s]}{\mathscr{D}_X[s] \cdot \theta_{f; \textbf{m}}(s)}
    \end{equation*}
    is a free $\mathscr{D}_{X,\mathfrak{x}}[s]$-resolution of $\mathscr{D}_X[s] / \mathscr{D}_X[s] \cdot \theta_{f; \textbf{m}}(s)$.
\end{proposition}

\begin{proof}
    We give enough of a sketch to explain how a modified proof of Proposition \ref{prop- free, Saito, Spencer resolution} applies here. It is enough to show that $\Sp_{f,\textbf{m}}^\bullet$ is acylic at $\mathfrak{x} \in X$. When $n=1$ then, working locally, we may assume that $\mathfrak{x} = 0$, $f = x \in \mathbb{C}[x]$, and $\mathbf{m} = m$. Then $\Der_{X}(-\log f; \mathbf{m})$ is generated by $x^{m} \partial_x$ and so $\theta_{f; \textbf{m}}(s)$ is generated by $x^m \partial_x - s$. To show $\Sp_{f; \textbf{m}}^\bullet$ is acyclic at $\mathfrak{x}$, it is enough to show the complex $0 \to \mathscr{D}_X[s] \to \mathscr{D}_X[s]$ induced by right multiplication with $x^m \partial_x - s$ is acylic, which is immediate. 

    Now assume the claim holds for $\mathbb{C}^t$ where $t < n$. Change coordinates to assume $\mathfrak{x} = 0$. Consider the intersection lattice of the arrangement $V(f)$. If $0$ belongs to a positive dimensional flat, then after a (linear) change of coordinates, we may find a $g \in \mathbb{C}[x_1, \dots, x_{n-1}] \subseteq \mathbb{C}[x_1, \dots, x_n] \ni f$ such that $f = (\text{unit}) \cdot g$ at $\mathfrak{x}$. In particular, 
    \begin{align*}
        \Der_{X, 0}(-\log f; \textbf{m}) 
        &= \{\delta \in \Der_{X,0} \mid \delta \bullet \ell_k \in \mathscr{O}_{X,0} \cdot \ell_k^{m_k} \text{ for all } k \} \\
        &= \{\delta \in \Der_X \mid \delta \bullet \ell_k \in \mathscr{O}_{X,0} \cdot \ell_k^{m_k} \text{ for all } k \text{ where } \Var(\ell_k) \subseteq \Var(g) \} \\ 
        &= \Der_{X, 0}(-\log g ; \textbf{m}_|g).
    \end{align*}
    Here $\textbf{m}_{|g} \in \mathbb{Z}_{\geq 1}^{\deg(g)}$ is induced by $\textbf{m}$: it only remembers the multiplicities $m_k$ corresponding to hyperplanes $V(\ell_k) \subseteq V(g)$. 

    We had been considering $g \in \mathbb{C}[x_1, \dots, x_n]$. When regarding it inside $\mathbb{C}[x_1, \dots, x_{n-1}]$ call it $h$. Using the product structure $\mathbb{C}^n \simeq \mathbb{C} \times \mathbb{C}^{n-1}$ with projections $\pr_i$ onto the $i^{\text{th}}$ factors, the above shows that 
    \begin{align*}
        \Der_{X,0}(-\log \pr_2^{\star} h; \textbf{m}_{| g}) 
        &= \Der_{X,0} (-\log f; \textbf{m}) \\
        &= \pr_1^\star \Der_{\mathbb{C},0} \oplus \pr_2^\star \Der_{\mathbb{C}^{n-1}, 0}(-\log h ; \textbf{m}_{|g}).
    \end{align*}
    The freeness of $\Der_{X,0}(-\log f; \textbf{m})$ is equivalent to the freeness of $\Der_{\mathbb{C}^{n-1},0}(-\log h; \textbf{m}_{|g})$. Hence the latter falls into our inductive set-up and $\Sp_{h; \mathbf{m}_{|g}}^\bullet$ must be acyclic. An argument entirely similar to that justifying Lemma \ref{lem-externalTensor} shows that the acylicity of $\Sp_{h; \mathbf{m}_{|g}}^\bullet$ guarantees the acylicity of $\Sp_{f; \mathbf{m}}^\bullet$ at $0$. 

    Thus, we may assume that $\Sp_{f; \mathbf{m}}^\bullet$ is acyclic at all points $\mathfrak{x} \in X$ not belonging to zero dimensional flats of the intersection lattice of $f$. That is, $\Sp_{f; \mathbf{m}}^\bullet$ is acyclic outside of a finite set of points. The ayclicity of $\Sp_{f; \mathbf{m}}^\bullet$ locally everywhere follows by Lemma \ref{lemma-noncommutative acyclicity analytic} (see also Remark \ref{rmk- analaytic vanish outside of points simplification}), just as in Proposition \ref{prop- free, Saito, Spencer resolution}. 
\end{proof}

\begin{remark}
    \noindent 
    \begin{enumerate}[label=(\alph*)]
        \item Recall the definition of $\theta_{f; \textbf{m}}(s)$. By replacing the $s$ in $\delta - s(\delta \bullet f)/f^{\textbf{m}}$ with a constant $\alpha \in \mathbb{C}$, one can assemble a Spencer-esque complex of left $\mathscr{D}_X$-modules similar to $\Sp_{f, \mathbf{m}}^\bullet$. Similarly by replacing the denominator $f^{\textbf{m}}$ with $f$. An argument akin to Proposition \ref{prop-SpencerAcyclicMultiDerivations} shows these new complexes are also acyclic. 
        \item One can check (an appropriately adjusted) Proposiiton \ref{prop-SpencerAcyclicMultiDerivations} holds in the algebraic category, essentially because the coordinate changes utilized in the induction are all linear. Proposition \ref{prop- free, Saito, Spencer resolution} requires the analytic category as the corresponding coordinate changes may be analytic.
        \item It is not clear to us what the relationship between $\mathscr{D}_X[s] / \mathscr{D}_X[s] \cdot \theta_{f; \mathbf{m}}(s)$ and $M_f^{\log}$ is. So we do not know what Proposition \ref{prop-SpencerAcyclicMultiDerivations} implies about $\mathscr{D}_X[s] f^s$.
    \end{enumerate}
\end{remark}

\bibliographystyle{abbrv}
\bibliography{refs}

\end{document}